\documentclass[11pt]{amsart}

\usepackage {enumerate}
\usepackage{setspace}
\usepackage{caption}
\usepackage{mathrsfs}
\usepackage{hyperref}
\usepackage{esint}
\usepackage{amssymb}
\usepackage{graphicx}
\usepackage{color}
\usepackage[all]{xy}

\newtheorem{theorem}{Theorem}[section]
\newtheorem{lemma}[theorem]{Lemma}
\newtheorem{corollary}[theorem]{Corollary}
\newtheorem{proposition}[theorem]{Proposition}

\numberwithin{equation}{section}

\theoremstyle {definition}
\newtheorem{definition}[theorem]{Definition}
\newtheorem{remark}[theorem]{Remark}

\usepackage[margin=1.1in]{geometry}

\DeclareMathOperator{\Ric}{Ric}
\begin{document}
\title[Rigidity of Area-Minimizing $2$-Spheres]{Rigidity of Area-Minimizing $2$-Spheres in $n$-manifolds with Positive Scalar Curvature}
\author{Jintian Zhu}
\address{Key Laboratory of Pure and Applied Mathematics, School of Mathematical Sciences, Peking University, Beijing, 100871, P.~R.~China}
\email{zhujt@pku.edu.cn}
\begin{abstract}
We prove that the least area of non-contractible immersed spheres is no more than $4\pi$ in any oriented compact manifold with dimension $n+2\leq 7$ which satisfies $R\geq 2$ and admits a map to $\mathbf S^2\times T^n$ with nonzero degree. We also prove a rigidity result for the equality case. This can be viewed as a generalization of the result in \cite{BBN2010} to higher dimensions.
\begin{flushleft}
\end{flushleft}
\end{abstract}
\renewcommand{\subjclassname}{\textup{2000} Mathematics Subject Classification}

\maketitle

\section{Introduction}
Throughout this paper, $(M^n,\bar g)$ will be an oriented closed Riemannian manifold with a nontrivial second homotopy group $\pi_2(M)$. We say that an immersed sphere $\Sigma$ in $M$ is non-contractilble, if it is not homotopic to a single point. Since $\pi_2(M)\neq \{0\}$, there exists at least one non-contractible immersed sphere. Denote
\begin{equation}
\mathcal F_M:=\left\{\Sigma\subset M\,:\,\Sigma\,\,\text{\rm is a non-contractible immersed sphere}\right\}
\end{equation}
and
\begin{equation}
\mathcal A_{\mathbf S^2}(M,\bar g):=\inf_{\Sigma\in\mathcal F_M} A_{\bar g}(\Sigma),
\end{equation}
where $A_{\bar g}(\Sigma)$ represents the area of $\Sigma$ with the induced metric from $\bar g$. We also use $R_{\bar g}$ to represent the scalar curvature of the metric $\bar g$ in the following context.


In the paper \cite{BBN2010}, H. Bray, S. Brendle and A. Neves proved the following result:
\begin{theorem}\label{BBN}
If $(M^3,\bar g)$ satisfies $R_{\bar g}\geq 2$, then we have
$
\mathcal A_{\mathbf S^2}(M,\bar g)\leq 4\pi.
$
Furthermore, the equality implies that the universal covering $(\hat M,\hat g)$ is $\mathbf S^2\times \mathbf R$.
\end{theorem}

It is natural to ask whether such result still holds for the dimension $n\geq 4$. Unfortunately, the result cannot be true in general. An easy example is as follows. Let $M=\mathbf S^2(r_1)\times \mathbf S^m(r_2)$ with $m\geq 3$ such that the scalar curvature of $M$ is $2$. With appropriate value for $r_2$ to be taken, the scalar curvature of $\mathbf S^n(r_2)$ can be arbitrarily close to $2$, from which it follows that the area of $\mathbf S^2(r_1)$ can be arbitrarily large. This yields that no upper bound depending only on the scalar curvature can be obtained for the least area of non-contractible spheres in this case.
Hopefully, if we impose further conditions on $(M^n,\bar g)$, such result can still be true. Namely, we have the following theorem:

\begin{theorem}\label{Thm: main1}
For $n+2\leq 7$, let $(M^{n+2},g)$ be an oriented closed Riemannian manifold with $R_g\geq 2$, which admits a non-zero degree map $f:M\to\mathbf S^2\times T^n$. Then $\mathcal A_{\mathbf S^2}(M,g)\leq 4\pi$. Furthermore, the equality implies that the universal covering $(\hat M,\hat g)$ is $\mathbf S^2\times \mathbf R^n$.
\end{theorem}

\begin{remark}
The restriction $n+2\leq 7$ comes from the regularity issue of area-minimizing hypersurface.
\end{remark}

Our inspiration for Theorem \ref{Thm: main1} comes from M. Gromov's work in \cite{Gro2018}. He proved that
\begin{theorem}[\cite{Gro2018}, P. 653, and P. 679]
If $(V^n,g)$ is an over-torical band with $R_g\geq n(n-1)$ and $n\leq 8$, then
$$
width(V,g)\leq \frac{2\pi}{n}.
$$
When the equality holds, the universal covering of the extreme band is $\mathbf R^{n-1}\rtimes O(n-1)$-invariant.
\end{theorem}
In the above theorem, an over-torical band means a Riemannian manifold $(V,g)$ with boundary $\partial V=\partial V_+\sqcup \partial V_-$, which admits a nonzero map $f:(V,\partial V_\pm)\to (T^{n-1}\times [-1,1],\, T^{n-1}\times\{\pm 1\})$.
The width of $(V,g)$ is defined to be the distance between $\partial V_+$ and $\partial V_-$. If one just takes $V$ to be the torical band $T^{n-1}\times [-1,1]$, then the theorem actually tells us that the lower positive bound of the scalar curvature gives an upper bound for the least distance transverse to $T^{n-1}$ slices.
From this point of view, our theorem deals with the area case in a similar spirit of M. Gromov's result.

We also recall that R. Schoen and S.T. Yau proved the following theorem:
\begin{theorem}[\cite{SY1979}, Corollary 2]
Let $M$ be a closed Riemannian manifold which admits a nonzero degree to $T^n$ with $n\leq 7$. Then the only possible metric with nonnegative scalar curvature is the flat one.
\end{theorem}
Based on this, our theorem seems to have a similar spirit of this result as well.
If one writes $T^n$ as $T^2\times T^{n-2}$ and considers the above theorem as one for the torus case, then our theorem can be thought as a quantitive one for the sphere case.

We have to say that Theorem \ref{Thm: main1} is one of the few rigidity results in high dimensions. Actually, the relationship of the scalar curvature and minimal $2$-surfaces is far from clear when the dimension of the ambient manifold is greater than 4, although there have been various results in $3$-dimensional case. Apart from Theorem \ref{BBN}, we mention that there are similar rigidity results for area-minimizing projective planes and min-max spheres in three dimensional manifolds. We refer readers to \cite{BBEN2010} and \cite{MN2012} for further information. It is a very interesting topic to prove some counterparts of these results in higher dimensions.

In high dimension case, there are some new challenges in the proof of Theorem \ref{Thm: main1}. To show the upper bound for the least area of non-contractible immersed spheres, it is no longer valid to analyzing the second variation of area for an area-minimizing sphere in the class of non-contractible spheres.
The first trouble is that an area-minimizing sphere may fail to be immersed due to possible branched points, which however can be ruled out in dimension 3. While even if the area-minimizing sphere is immersed, we cannot obtain enough control on its area from the second variation formula due to higher codimensions.

To overcome these challenges, instead of taking an area-minimizing sphere, we construct directly a non-contractible embedded one with area no more than $4\pi$. To be precise, we apply the Torical Symmetrization argument from \cite{Gro2018} to find a slicing (see Definition \ref{Defn: slicing})
\begin{equation*}
(\Sigma_2,g_2)\looparrowright(\Sigma_3,g_3)\looparrowright\cdots\looparrowright (\Sigma_{n+1},g_{n+1}) \looparrowright (\Sigma_{n+2},g_{n+2}) = (M,\bar g),
\end{equation*}
where $\Sigma_2$ appears to be a non-contractible embedded sphere and $\Sigma_2\times T^n$ admits a warped product metric
$$
\bar g_2=g_2+\sum_{i=2}^{n+1}u_i^2\mathrm d t_{i+1}^2
$$
with $R_{\bar g_2}\geq 2$. In this case, we are able to show $A_{g_2}(\Sigma_2)\leq 4\pi$ from a straightforward calculation, which then gives $\mathcal A_{\mathbf S^2}(M,\bar g)\leq 4\pi$.
To present a more clear picture, let me briefly illustrate how to obtain a desired slicing from the Torical Symmetrization argument. The essential step is as follows. Suppose that we are given a Riemannian manifold $(\bar M,\bar g)$, which admits a nonzero degree map to $\mathbf S^2\times T^n$, then from geometric measure theory we can find an area-minimizing hypersurface $(\tilde M,\tilde g)$ in the pull-back homology class of $\mathbf S^2\times T^{n-1}$. Taking the first eigenfunction $\tilde u$ of the Jacobi operator $\tilde{\mathcal J}$, we can obtain a new manifold $(\bar M',\bar g')$ with $\bar M'=\tilde M\times\mathbf S^1$ and $\bar g'=\tilde u^2\mathrm dt^2+\tilde g$, which satisfies $R_{\bar g'}\geq R_{\bar g}$. The advantage of this construction is a lift in symmetry. That is, if $(M,\bar g)$ is $T^m$-symmetric, then $(\bar M',\bar g')$ is $T^{m+1}$-symmetric. Using the step stated above for several times and following the procedure as shown below, we can obtain the desired slicing.
\begin{figure}[htbp]
\begin{displaymath}
\xymatrix{
  (\bar\Sigma_{k+1},\bar g_{k+1}) \ar[rr]^{\quad\text{\rm Minimizing Procedure}}
                &  &    (\tilde\Sigma_k,\tilde g_k) \ar[dl]^{\,\,\,\text{\rm Warped Produt}}    \\
                & (\bar\Sigma_k,\bar g_k)   \ar[ul]^{\text{\rm Replace}}              }
\end{displaymath}
\caption{\small We start with $(\bar\Sigma_{n+2},\bar g_{n+2})=(M,\bar g)$ to find an area-minimizing surface $(\tilde\Sigma_{n+1},\tilde g_{n+1})$. From $(\tilde\Sigma_{n+1},\tilde g_{n+1})$ we constrct the warped product $(\bar\Sigma_{n+1},\bar g_{n+1})$ with $\bar\Sigma_{n+1}=\tilde\Sigma_{n+1}\times\mathbf S^1$, and then we use this as a new starting manifold to find another area-minimizing surface $(\tilde \Sigma_{n},\tilde g_n)$. By warped product, we obtain $(\bar\Sigma_n,\bar g_n)$. ... After repeating again and again, we obtain $\bar\Sigma_k$ for all $k\geq 2$. From higher and higher symmetry, we can write $\bar\Sigma_k=\Sigma_k\times T^{n+2-k}$, and these $\Sigma_k$ provide us the desired slicing.}
\end{figure}
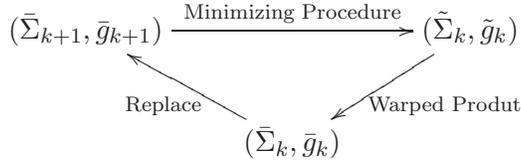

The proof for the rigidity in the equality case turns out to be very subtle. From $\mathcal A_{\mathbf S^2}(M,\bar g)=4\pi$, it is easy to show that $(\Sigma_2,g_2)$ is isometric to the standard sphere, therefore it is enough to show that there exist local isometries $\Phi_k:\Sigma_k\times\mathbf R\to\Sigma_{k+1}$ for $k$ from $2$ to $n+1$.
To obtain this result, we choose to prove a stronger one that there exists a local isometry $\tilde\Phi_k:\tilde\Sigma_k\times\mathbf R\to\bar\Sigma_{k+1}$. For simplicity, we will only explain our idea in the case $k=2$, since the rest cases follow from the same argument.
Through a calculation, it is direct to show $R_{\bar g_2}=R_{\tilde g_2}=2$ from $\mathcal A_{\mathbf S^2}(M,\bar g)=4\pi$.
From the construction, this implies that $\tilde\Sigma_2$ is totally geodesic and has vanished normal Ricci curvature in $\bar\Sigma_3$.
To show the existence of $\tilde\Phi_2$, we adopt the idea from \cite[Theorem 1.6]{CCE2016} to construct an area-minimizing foliation around $\tilde\Sigma_2$. From the inverse function theorem, it is not difficult to show the existence of a foliation $\{\tilde\Sigma_{2,t}\}_{-\epsilon\leq t\leq\epsilon}$ with constant mean curvature.
The difficulty here is to show that the mean curvatures $\tilde H_{2,t}$ have the correct sign such that each slice in the foliation is also area-minimizing.
At this stage, we follow the idea from \cite[page 679]{Gro2018} to rule out the worried situation. Assuming that one of the mean curvatures $\tilde H_{2,t}$ has a wrong sign, such as $\tilde H_{2,t_0}>0$ for $t_0>0$, by minimizing an appropriate brane functional in the from of $\mathcal B=A(\tilde\Sigma)-\delta V(\tilde\Omega)$, we can find a soap bubble $(\hat\Sigma_k,\hat g_k)$ in $\bar\Sigma_{k+1}$ such that $\hat\Sigma_k\times\mathbf S^1$ has a warped product metric
\begin{equation*}
\hat g_{k+1}=\hat g_k+\hat u_k^2\mathrm dt_{k+1}^2
\end{equation*}
with $R_{\hat g_{k+1}}\geq 2+\delta^2$. With the Torical Symmetrization argument applied to $\hat\Sigma_k\times\mathbf S^1$, there will be a non-contractible embedded sphere in $(M,\bar g)$ with area strictly less than $4\pi$, which leads to a contradiction. Consequently, each surface in the constructed foliation must be area-minimizing. Therefore they are totally geodesic and have vanished normal Ricci curvature as well. From this, it is quick to show that $\bar\Sigma_3$ splits as $\tilde\Sigma_2\times(-\epsilon,\epsilon)$ around $\tilde\Sigma_2$. Through a continuous argument, we can obtain the local isometry $\tilde\Phi_2$.

In the following, the article will be organized as follows. In section 2, we use the torical symmetrization argument to construct a non-contractible embedded sphere with area no more than $4\pi$. In section 3, we prove the rigidity for the case of equality.
\mbox{} \\

\noindent {\bf Acknowledgments.} This research is partially supported by the NSFC grants No. 11671015 and 11731001. The author also would like to thank Professor Yuguang Shi for many encouragements and discussions.
\section{Torical Symmetrization}
In this section, we show how to use Torical Symmetrization argument to construct an embedded non-contractible sphere with area no more than $4\pi$. For the convenience to state our results, we introduce the following definition first.
\begin{definition}\label{Defn: slicing}
A sequence of isometric immersions
\begin{equation*}
(\Sigma_0,g_0)\stackrel{\phi_1}\looparrowright (\Sigma_{1},g_{1})\stackrel{\phi_2}\looparrowright\cdots\stackrel{\phi_{m}}\looparrowright(\Sigma_m,g_m)
\end{equation*}
is called a slicing if each immersion $\phi_i$ is codimension one for $i=1,2,\ldots,m$.
\end{definition}
In the following, we denote the projection map from $\mathbf S^2\times T^i$ to $\mathbf S^2\times T^j$ by $\pi^i_j$ for any $i\geq j$. Using the Torical Symmetrization argument, we can obtain the following proposition:
\begin{proposition}\label{Prop: slicing}
For $n+2\leq 7$, let $(M^{n+2},\bar g)$ be an oriented closed Riemannian manifold with $R_{\bar g}\geq 2$, which admits a non-zero degree map $f:M\to\mathbf S^2\times T^n$. Then there exists a slicing
\begin{equation*}
(\Sigma_2\approx\mathbf S^2,g_2)\looparrowright(\Sigma_3,g_3)\looparrowright\cdots\looparrowright (\Sigma_{n+1},g_{n+1}) \looparrowright (\Sigma_{n+2},g_{n+2})=(M,\bar g),
\end{equation*}
such that the map $f_k=\pi^{n}_{k-2}\circ f|_{\Sigma_k}:\Sigma_k\to\mathbf S^2\times T^{k-2}$ has non-zero degree. Furthermore, there exist positive functions $u_k:\Sigma_k\to\mathbf R,\,k=2,3,\ldots,n+1,$ such that the manifolds $(\bar\Sigma_k,\bar g_k)$ with $\bar\Sigma_k=\Sigma_k\times T^{n+2-k}$ and
\begin{equation}
\bar g_k=g_k+\sum_{i=k}^{n+1}(u_i|_{\Sigma_k})^2\mathrm dt_{i+1}^2
\end{equation}
satisfies $R_{\bar g_k}\geq 2$, and that $(\tilde\Sigma_k,\tilde g_k)$ is area-minimizing\footnote{In this paper, area-minimizing means to have the least area in the homotopy class, although the hypersurface may have some much stronger area-minimizing property.} in $(\bar\Sigma_{k+1},\bar g_{k+1})$, where $\tilde\Sigma_k=\Sigma_k\times T^{n+1-k}$ and
\begin{equation*}
\tilde g_k=g_k+\sum_{i=k+1}^{n+1}(u_i|_{\Sigma_k})^2\mathrm dt_{i+1}^2.
\end{equation*}
\end{proposition}
\begin{proof}
We are going to construct the desired slicing with an induction argument. For convenience, let us denote $(\Sigma_{n+2},\bar g_{n+2},u_{n+2},f_{n+2})=(M,\bar g,1,f)$. Without loss of generality, we assume that $f$ is a smooth function. Now we state how to obtain $(\Sigma_{k},g_{k},u_k,f_{k})$ from $(\Sigma_{k+1},g_{k+1},u_{k+1},f_{k+1})$ for $k=n+1,n,\ldots,2$. Define the map
\begin{equation*}
F_{k+1}=(f_{k+1},id)\,:\,\Sigma_{k+1}\times T^{n+1-k}\to \mathbf S^2\times T^n=\mathbf S^2\times T^{k-1}\times T^{n+1-k}.
\end{equation*}
According to Sard's theorem, there exists a $\theta_{k-1}$ such that the preimage
\begin{equation*}
S=F_{k+1}^{-1}(\mathbf S^2\times T^{k-2}\times\{\theta_{k-1}\}\times T^{n+1-k})=f_{k+1}^{-1}(\mathbf S^2\times T^{k-2}\times\{\theta_{k-1}\})\times T^{n+1-k}
\end{equation*}
is a smooth $2$-sided embedded hypersurface.
It is easy to see that $\deg \pi^{n}_{n-1}\circ F_{k+1}|_S=\deg f_{k+1}\neq 0$, from which we know that $[S]$ represents a nontrivial homology class.
By geometric measure theory, there exists a smooth embedded, $2$-sided, area-minimizing hypersurface $(\tilde\Sigma_{k},\tilde g_k)$ in the homology class $[S]$ in Riemannian manifold $(\bar\Sigma_{k+1},\bar g_{k+1})$, where $\tilde g_k$ is the induced metric and
\begin{equation*}
\bar\Sigma_{k+1}=\Sigma_{k+1}\times T^{n+1-k},\quad\bar g_{k+1}=g_{k+1}+\sum_{i=k+1}^{n+1}(u_i|_{\Sigma_{k+1}})^2\mathrm dt_{i+1}^2.
\end{equation*}

We claim that $\tilde\Sigma_k$ is $T^{n+1-k}$-invariant. Notice that $(\bar\Sigma_{k+1},\bar g_{k+1})$ is $T^{n+1-k}$-invariant, taking isometric $\mathbf S^1$-action $\rho_{i,\theta}$, $i=k+2, k+3,\ldots,n+2$, such that
\begin{equation*}
\rho_{i,\theta}\left((x,t_{k+2},\ldots,t_i,\ldots,t_{n+2})\right) =(x,t_{k+2},\ldots,t_i+\theta,\ldots,t_{n+2}),\quad x\in \Sigma_{k+1},
\end{equation*}
then $\rho_{i,\theta}$ induces a smooth variation vector field $X=\phi\nu$ on $\tilde\Sigma_k$, where $\nu$ is a fixed unit normal vector field.
If $\phi$ is not identical to zero, from the fact that $\tilde\Sigma_k$ is area-minimizing and $\rho_{i,\theta}$ is isometry, we see that $\phi$ is the first eigenfunction of the Jacobi operator $\tilde{\mathcal J}_k$, which implies that $\phi$ has a fixed sign.
This yields that the algebraic intersection number of the circle $\mathbf S^1_i$ and $\tilde\Sigma_k$ is nonzero.
However, notice that $\mathbf S^1_i$ lies on the hypersurface $S$, the algebraic intersection of $\mathbf S_i^1$ and $S$ is zero, which leads to a contradiction. Therefore, we have $\phi\equiv 0$ and that $\tilde\Sigma_k$ is $\rho_{i,\theta}$-invariant.
Ranging $i$ from $k+2$ to $n+2$, we know that $\tilde\Sigma_k$ is $T^{n+1-k}$-invariant, and it follows that $\tilde\Sigma_k=\Sigma_k\times T^{n+1-k}$.
Denote $f_k=\pi^{n}_{k-2}\circ f|_{\Sigma_k}=\pi^{k-1}_{k-2}\circ f_{k+1}|_{\Sigma_k}$.
Notice that $\tilde\Sigma_k$ is homologic to $S$, we have
\begin{equation*}
\deg f_k=\deg \pi^n_{n-1}\circ F_{k+1}|_{\tilde\Sigma_k}=\deg \pi^n_{n-1}\circ F_{k+1}|_S=\deg F_{k+1}=\deg f_{k+1}\neq 0.
\end{equation*}
It is possible that $\Sigma_k$ have several connected components, but we can choose one of these components, still denoted by $\Sigma_k$, satisfying $\deg f_k\neq 0$. In this case, it follows from \cite[Lemma 33.4]{S1983} that $\tilde\Sigma_k=\Sigma_k\times T^{n+1-k}$ is still area-minimizing. Let $g_k$ be the induced metric of $\Sigma_k$ from $(\Sigma_{k+1},g_{k+1})$, then
\begin{equation*}
\tilde g_k=g_k+\sum_{i=k+1}^{n+1} (u_i|_{\Sigma_k})^2\mathrm dt_{i+1}^2,
\end{equation*}
and the last statement in the proposition follows easily.

In the following, we construct the desired function $u_k$. Since $\tilde\Sigma_k$ is area-minimizing, the Jacobi operator
\begin{equation}\label{Eq: Jacobi-Operator}
\tilde{\mathcal J}_k=-\Delta_{\tilde g_k}-(\Ric_{\bar g_{k+1}}(\tilde\nu_k,\tilde \nu_k)+\|\tilde A_k\|^2)=-\Delta_{\tilde g_k}-\frac{1}{2}\left(R_{\bar g_{k+1}}|_{\tilde\Sigma_k}- R_{\tilde g_k}+\|\tilde A_k\|^2\right)
\end{equation}
is nonnegative, where $\tilde\nu_k$ and $\tilde A_k$ are the unit normal vector field and the corresponding second fundamental form. Denote $\tilde u_k$ to be the first eigenfunction of $\tilde{\mathcal J}_k$ with respect to the first eigenvalue $\tilde \lambda_{1,k}\geq 0$. Since the first eigenspace is dimension one, $\tilde u_k$ is also $T^{n+1-k}$-invariant, and hence $\tilde u_k$ can be viewed as a function $u_k$ over $\Sigma_k$. Given the smooth metric
\begin{equation*}
\bar g_k=g_k+\sum_{i=k}^{n+1}(u_i|_{\Sigma_k})^2\mathrm dt_{i+1}^2
\end{equation*}
on $\bar\Sigma_k$, direct calculation shows
\begin{equation}\label{Eq: Monotonicity Scalar Curvature}
R_{\bar g_k}= R_{\tilde g_k}-2\frac{\Delta_{\tilde g_k}\tilde u_k}{\tilde u_k}=R_{\bar g_{k+1}}|_{\tilde\Sigma_k}+\|\tilde A_k\|^2+2\tilde\lambda_{1,k}\geq 2.
\end{equation}

Through an induction argument from $k=n+1$ to $k=2$, we can obtain the desired slicing
\begin{equation*}
(\Sigma_2,g_2)\looparrowright(\Sigma_3,g_3)\looparrowright\cdots\looparrowright (\Sigma_{n+1},g_{n+1}) \looparrowright (\Sigma_{n+2},g_{n+2})=(M,\bar g).
\end{equation*}
The only thing needs to be verified is that $\Sigma_2$ is a $2$-sphere. Otherwise, $\Sigma_2$ is a surface with genus $\mathfrak g\geq 1$, and then there exists a map $\tilde f:\Sigma_2\times T^n\to T^{n+2}$ with nonzero degree. Combined with  \cite[Corollary 2]{SY1979}, this contradicts to the fact $R_{\bar g_2}\geq 2$.
\end{proof}

The following lemma yields an upper bound for the area of $\Sigma_2$.
\begin{lemma}\label{Lem: area-bound}
Let $(\mathbf S^2\times T^n,\bar g)$ be a Riemannian manifold with $R_{\bar g}\geq 2$ and
\begin{equation}
\bar g=g+\sum_{i=2}^{n+1} u_i^2\mathrm d t_{i+1}^2,
\end{equation}
where $g$ is a smooth metric on $\mathbf S^2$ and $u_i:\mathbf S^2\to\mathbf R$ are positive smooth functions for $i=2,3,\ldots,n+1$.
Then we have
$$
A_g(\mathbf S^2)\leq 4\pi,
$$
where equality holds if and only if $(\mathbf S^2,g)$ is isometric to the standard sphere and $u_i$ are positive constants for $i=2,3,\ldots,n+1$.
\end{lemma}
\begin{proof}
Through direct calculation (refer to \cite[Lemma 2.5]{SY2017}), we have
\begin{equation*}
R_{\bar g}=R_g-2\sum_{i=2}^{n+1} u_i^{-1}\Delta_g u_i-2\sum_{2\leq i<j\leq n+1}\langle \nabla_g\log u_i,\nabla_g\log u_j\rangle.
\end{equation*}
Integrating over $\mathbf S^2$, we obtain
\begin{equation}\label{Eq: area-bound}
2A_g(\mathbf S^2)\leq\int_{\mathbf S^2}R_{\bar g}\,\mathrm d\mu_g=8\pi-\sum_{i=2}^{n+1}\int_{\mathbf S^2}|\nabla_g\log u_i|^2\,\mathrm d\mu_g-\int_{\mathbf S^2}\left|\sum_{i=2}^{n+1}\nabla_g\log u_i\right|^2\,\mathrm d\mu_g,
\end{equation}
where we have used the fact
\begin{equation*}
u_i^{-1}\Delta_gu_i=\Delta_g\log u_i+|\nabla_g\log u_i|^2,\quad i=2,3,\ldots,n+1.
\end{equation*}
Then it is clear that $A_g(\mathbf S^2)\leq 4\pi$. When the equality holds, we see that $u_i$ are positive constants for $i=2,3,\ldots,n+1$, and then $R_g=R_{\bar g}\geq 2$. From the Gauss-Bonnet formula, we obtain $R_g\equiv 2$, which implies that $(\mathbf S^2,g)$ is isometric to the standard sphere.
\end{proof}

\begin{corollary}\label{Cor: area-bound}
Let $(M,\bar g)$ be as in Theorem \ref{Thm: main1}. Then we have $\mathcal A_{\mathbf S^2}(M,\bar g)\leq 4\pi$.
\end{corollary}
\begin{proof}
From Proposition \ref{Prop: slicing}, we can find an embedded sphere $\Sigma_2$ such that $\pi^n_0\circ f|_{\Sigma_2}:\Sigma_2\to\mathbf S^2$ has nonzero degree. This implies that $\Sigma_2$ is non-contractible. By Lemma \ref{Lem: area-bound}, we know $A_{\bar g}(\Sigma_2)\leq 4\pi$. Therefore, $\mathcal A_{\mathbf S^2}(M,\bar g)\leq A_{\bar g}(\Sigma_2)\leq 4\pi$.
\end{proof}

\section{The Case of Equality}
In this section, $(M^{n+2},\bar g)$ is always an oriented closed Riemannian manifold with dimension $n+2\leq 7$ and scalar curvature $R_{\bar g}\geq 2$, which admits a non-zero degree map $f:M\to\mathbf S^2\times T^n$. Also, we use
\begin{equation*}
(\Sigma_2,g_2)\looparrowright(\Sigma_3,g_3)\looparrowright\cdots\looparrowright (\Sigma_{n+1},g_{n+1}) \looparrowright (\Sigma_{n+2},g_{n+2})=(M,\bar g)
\end{equation*}
to denote an arbitrary slicing constructed as in Proposition \ref{Prop: slicing}, and we adapt the same notation there.

The first lemma below is the starting point for our rigidity result.
\begin{lemma}\label{Lem: tilde Sigma 2}
If $\mathcal A_{\mathbf S^2}(M,\bar g)=4\pi$, then $(\Sigma_2,g_2)$ is isometric to the standard sphere, $R_{\bar g_2}=2$ and $u_i|_{\Sigma_2}$ are positive constants for $i=2,3,\ldots,n+1$.
\end{lemma}
\begin{proof}
It follows from Proposition \ref{Prop: slicing} that $(\bar\Sigma_2,\bar g_2)$ satisfies $R_{\bar g_2}\geq 2$, where
$$
\bar\Sigma_2=\Sigma_2\times T^n\quad \text{\rm and}\quad \bar g_2=g_2+\sum_{i=2}^{n+1} (u_i|_{\Sigma_2})^2\mathrm dt_{i+1}^2.
$$
If $\mathcal A_{\mathbf S^2}(M,\bar g)=4\pi$, from Lemma \ref{Lem: area-bound}, we have $A_{g_2}(\Sigma_2)=4\pi$. As a result, $(\Sigma_2,g_2)$ is isometric to the standard sphere, and $u_i|_{\Sigma_2}$ are positive constants. Therefore, we have $R_{\bar g_2}=2$.
\end{proof}
In the following, we prove the infinitesimal rigidity for $\tilde \Sigma_k$.
\begin{lemma}\label{Lem: tilde Sigma k}
Fix some $2\leq k\leq n+1$, if $R_{\bar g_k}\equiv 2$ and $u_i|_{\Sigma_k}$ are positive constants for $i=k,k+1,\ldots,n+1$, then $\tilde\Sigma_k$ is totally geodesic in $\bar\Sigma_{k+1}$ and satisfies $\Ric_{\bar g_{k+1}}(\tilde \nu_k,\tilde \nu_k)=0$.
\end{lemma}
\begin{proof}
Recall from (\ref{Eq: Monotonicity Scalar Curvature}) that
\begin{equation*}
R_{\bar g_k}\geq R_{\bar g_{k+1}}|_{\tilde\Sigma_k}+\|\tilde A_k\|^2+2\tilde\lambda_{1,k},
\end{equation*}
combined with the facts $R_{\bar g_k}=2$, $R_{\bar g_{k+1}}\geq 2$ and $\tilde\lambda_{1,k}\geq 0$, we obtain $R_{\bar g_{k+1}}|_{\tilde\Sigma_k}\equiv 2$, $\tilde A_k\equiv 0$ and $\tilde\lambda_{1,k}=0$. Since $u_k$ is a positive constant on $\Sigma_k$, we also know $R_{\tilde g_k}\equiv 2$, which implies $\Ric_{\bar g_{k+1}}(\tilde\nu_k,\tilde\nu_k)=0$.
\end{proof}

As in \cite[Proposition 5]{BBN2010}, we use the infinitesimal rigidity of $\tilde\Sigma_k$ to construct a foliation of constant mean curvature hypersurfaces around $\tilde\Sigma_k$ in $\bar\Sigma_{k+1}$.

\begin{lemma}\label{Lem: foliation}
If $\tilde\Sigma_k$ is an area-minimizing hypersurface in $\bar\Sigma_{k+1}$ with vanished second fundamental form and normal Ricci curvature, then we can construct a local foliation $\{\tilde\Sigma_{k,t}\}_{-\epsilon\leq t\leq\epsilon}$ in $\bar\Sigma_{k+1}$ such that $\tilde\Sigma_{k,t}$ are of constant mean curvatures and $\tilde\Sigma_{k,0}=\tilde\Sigma_k$. Furthermore, $\tilde\Sigma_{k,t}$ is $T^{n+1-k}$-invariant.
\end{lemma}
\begin{proof}
Consider the map
\begin{equation}
\psi: C^{2,\alpha}(\tilde\Sigma_k)\to \mathring C^{\alpha} (\tilde\Sigma_k)\times\mathbf R,\quad f\mapsto (\tilde H_f-\fint_{\tilde\Sigma_k}\tilde H_f\,\mathrm d\mu,\fint_{\tilde\Sigma_k}f\,\mathrm d\mu),
\end{equation}
where $\tilde H_f$ is the mean curvature of the graph over $\tilde\Sigma_k$ with graph function $f$ and the space
$$\mathring C^\alpha(\tilde \Sigma_k)=\{\phi\in C^\alpha(\tilde\Sigma_k)\,:\,\int_{\tilde\Sigma_k}\phi\,\mathrm d\mu=0\}.$$
Since $\Ric_{\bar g_{k+1}}(\tilde\nu_k,\tilde\nu_k)=0$ and $\tilde A_k=0$, it is clear that the linearized operator of $\psi$ at $f=0$
\begin{equation}
D\psi|_0\,:\,C^{2,\alpha}(\tilde\Sigma_k)\to \mathring C^\alpha(\tilde\Sigma_k)\times \mathbf R,\quad  g\mapsto(-\Delta g,\fint_{\tilde\Sigma_k}g\,\mathrm d\mu)
\end{equation}
is invertible. From the inverse function theorem, we can find a family of function $f_t:\tilde\Sigma_k\to\mathbf R$ with $t\in(-\epsilon,\epsilon)$ with the following properties:
\begin{itemize}
\item The functions $f_t$ satisfies $f_0\equiv 0,$
\begin{equation}\label{Eq: graphs ft}
\left.\frac{\partial}{\partial t}\right|_{t=0} f_t\equiv 1,\quad\text{\rm and}\quad\fint_{\tilde\Sigma_k}f_t\,\mathrm d\mu=t.
\end{equation}
\item The graphs $\tilde\Sigma_{k,t}$ over $\tilde\Sigma_k$ with graph function $f_t$ has constant mean curvature.
\end{itemize}
From (\ref{Eq: graphs ft}), with the value of $\epsilon$ decreased a little bit, the speed $\partial_tf_t$ will be positive everywhere, from which it follows that the graphs $\tilde\Sigma_{k,t}$ form a foliation around $\tilde\Sigma_k$ for $-\epsilon\leq t\leq\epsilon$. The $T^{n+1-k}$-invariance of $\tilde \Sigma_{k,t}$ is due to the uniqueness from the inverse function theorem.
\end{proof}

We are ready to prove the existence of local isometries $\tilde\Phi_k$.
\begin{proposition}\label{Prop: rigidity slicing}
If $\mathcal A_{\mathbf S^2}(M,\bar g)=4\pi$, then for any slicing constructed as in Proposition \ref{Prop: slicing} and any $2\leq k\leq n+1$, we have $R_{\bar g_k}\equiv 2$ and that $u_i|_{\Sigma_k}$ are positive constants for $i$ from $k$ to $n+1$. In addition, there exists a local isometry $\tilde\Phi_k:\tilde\Sigma_k\times\mathbf R\to \bar\Sigma_{k+1}$.
\end{proposition}
\begin{proof}
From Lemma \ref{Lem: tilde Sigma 2}, for any slicing constructed as in Proposition \ref{Prop: slicing}, we have $R_{\bar g_2}\equiv 2$ and that $u_i|_{\Sigma_2}$ are positive constants for $i$ from $2$ to $n+1$.
For the rest cases, we are going to apply the induction argument.
Assuming that we already obtain $R_{\bar g_{k}}\equiv 2$ and that $u_i|_{\Sigma_{k}}$ are positive constants for $i$ from $k$ to $n+1$, using Lemma \ref{Lem: tilde Sigma k}, we know that $\tilde\Sigma_{k}$ is totally geodesic and has vanished normal Ricci curvature in $\bar\Sigma_{k+1}$.
From Lemma \ref{Lem: foliation}, there exists a constant mean curvature foliation $\{\tilde\Sigma_{k,t}\}_{-\epsilon\leq t\leq \epsilon}$ in $\bar\Sigma_{k+1}$ with $\tilde\Sigma_{k,0}=\tilde\Sigma_k$.
Denote $\tilde H_{k,t}$ to be the mean curvature of $\tilde \Sigma_{k,t}$, we claim that $\tilde H_{k,t}\equiv 0$ for any $t\in(-\epsilon,\epsilon)$.
It suffices to prove this for nonnegative $t$, since the other side follows from the same argument.
Suppose that the consequence is not true, then there exists a $t_0$ such that $\tilde H_{k,t_0}>0$, since otherwise $\tilde H_{k,t}\leq 0$ and the area-minimizing property of $\tilde\Sigma_t$ will imply $\tilde H_{k,t}\equiv 0$. Denote $\tilde\Omega_{k,t_0}$ to be the region enclosed by $\tilde\Sigma_{k,t_0}$ and $\tilde\Sigma_{k}$, and define the brane functional
\begin{equation}
\mathcal B(\tilde\Omega)=A({\partial\tilde\Omega}\backslash\tilde\Sigma_k)-\delta V(\tilde\Omega),\quad 0<\delta < \tilde H_{k,t_0},
\end{equation}
where $\tilde\Omega$ is any Borel subset of $\tilde\Omega_{k,t_0}$ with finite perimeter and $\tilde\Sigma_k\subset\partial\tilde\Omega$. Due to $0<\delta<\tilde H_{k,t_0}$, the hypersurfaces $\tilde\Sigma$ and $\tilde\Sigma_{k,t_0}$ serve as barriers, therefore we can find a Borel set $\hat\Omega_{k,t_0}$ minimizing the brane functional $\mathcal B$, for which $\hat\Sigma_{k}=\partial\hat\Omega_{k,t_0}\backslash\tilde\Sigma_k$ is a smooth $2$-sided hypersurface disjoint from $\tilde\Sigma_k$ and $\tilde\Sigma_{k,t_0}$.
A similar argument as in Proposition \ref{Prop: slicing} gives $\hat\Sigma_k$ is $T^{n+1-k}$-invariant, so we can write $\hat\Sigma_k$ as $\hat\Sigma_k'\times T^{n+1-k}$, where $\hat\Sigma_k'$ is a hypersurface in $(\Sigma_{k+1},g_{k+1})$. Denote the induced metric of $\hat\Sigma_k'$ by $\hat g_k'$, then the induced metric of $\hat\Sigma_k$ is
\begin{equation*}
\hat g_k=\hat g_k'+\sum_{i=k+1}^{n+1}(u_i|_{\hat\Sigma_k'})^2\mathrm dt_{i+1}^2.
\end{equation*}
Since $\hat\Sigma_{k}$ is $\mathcal B$-minimizing, it has constant mean curvature $\delta$ and it is $\mathcal B$-stable. So the Jacobi operator
\begin{equation}
\hat{\mathcal J}_k=-\Delta_{\hat g_k}-(\Ric_{\bar g_{k+1}}(\hat\nu_k,\hat\nu_k)+\|\hat A_k\|^2)=-\Delta_{\hat g_k}-\frac{1}{2}(R_{\bar g_{k+1}}-R_{\hat g_k}+\delta^2+\|\hat A_k\|^2)
\end{equation}
is nonnegative, where $\hat\nu_k$ is the unit normal vector field of $\hat\Sigma_k$ and $\hat A_k$ is the corresponding second fundamental form. Taking the first eigenfunction $\hat u_{k}$ of $\hat{\mathcal J}_k$ with respect to the first eigenvalue $\hat\lambda_{1,k}\geq 0$ and defining the following metric
\begin{equation}
\hat g_{k+1}=\hat g_k+\hat u_k^2\mathrm dt_{k+1}^2,
\end{equation}
on $\hat\Sigma_{k}\times \mathbf S^1$, we obtain
$$
R_{\hat g_{k+1}}\geq R_{\bar g_{k+1}}+\delta^2+\|\hat A_k\|^2+2\hat\lambda_{1,k}\geq 2+\delta^2.
$$
Since $\hat\Sigma_k$ is homologic to $\tilde\Sigma_k$, the map $\hat F_k=\pi^n_{n-1}\circ F_{k+1}|_{\hat\Sigma_k}$ has nonzero degree. If $\hat\Sigma_k$ is not connected, we may choose a suitable component as $\hat\Sigma_k$ such that $\hat F_k$ has non-zero degree. As a result, the map $\hat f_k=\pi^n_{k-2}\circ f|_{\hat\Sigma_k'}$ has nonzero degree. Repeating what we have done as in Proposition \ref{Prop: slicing}, we can find a slicing
\begin{equation*}
(\hat\Sigma_2',\hat g_2')\looparrowright(\hat\Sigma_3',\hat g_3')\looparrowright\cdots\looparrowright (\hat\Sigma_{k}',\hat g_k')\looparrowright (\Sigma_{k+1},g_{k+1})\looparrowright\cdots\looparrowright (\Sigma_{n+2},g_{n+2})=(M,\bar g)
\end{equation*}
and functions $\hat u_i\,:\,\hat\Sigma_i'\to\mathbf R$, $i=2,3,\ldots,k$, such that the scalar curvature $R_{\hat g_2}\geq 2+\delta^2$ for
\begin{equation*}
\hat g_2=\hat g_2'+\sum_{i=2}^k (\hat u_i|_{\hat\Sigma_2'})^2\mathrm dt_{i+1}^2+\sum_{i=k+1}^{n+1}(u_i|_{\hat\Sigma_2'})^2\mathrm dt_{i+1}^2.
\end{equation*}
Also, we have $\deg \pi^n_0\circ f|_{\hat\Sigma_2}\neq 0$. Therefore, $\hat\Sigma_2$ is a non-contractible embedded sphere in $M$. It follows from Lemma \ref{Lem: area-bound} that $A(\hat\Sigma_2)\leq 4\pi(1+\delta^2/2)^{-1}<4\pi$, which leads to a contradiction.

Since $\tilde H_{k,t}$ vanishes for any $t\in(-\epsilon,\epsilon)$, there holds $A(\tilde\Sigma_{k,t})\equiv A(\tilde\Sigma_k)$, which yields that $\tilde\Sigma_{k,t}$ is also area-minimizing.
Replacing $\tilde\Sigma_k$ by $\tilde\Sigma_{k,t}$, we can also obtain a slicing as in Proposition \ref{Prop: slicing}.
Combined with the induction hypothesis and Lemma \ref{Lem: tilde Sigma k}, $\tilde \Sigma_{k,t}$ is totally geodesic and has vanished normal Ricci curvature.

In the following, we show that there exists a local isometry $\tilde\Phi_k:\tilde\Sigma_k\to\bar\Sigma_{k+1}$. Denote $\tilde V_k$ to be the normal variation vector field of the foliation $\tilde\Sigma_{k,t}$ and $\tilde\Phi_{k}:\tilde\Sigma_k\times (-\epsilon,\epsilon)\to\bar\Sigma_{k+1}$ to be the flow generated by $\tilde V_k$.
Notice that $\tilde\Sigma_{k,t}$ are $T^{n+1-k}$-invariant, $\tilde\Sigma_{k,t}$ can be written as $\Sigma_{k,t}\times T^{n+1-k}$ and $\tilde V_k$ can be viewed as the normal variation vector field $V_k$ of $\Sigma_{k,t}$ on $\Sigma_{k+1}$.
Let $\Phi_k:\Sigma_k\times(-\epsilon,\epsilon)\to\Sigma_{k+1}$ be the flow generated by $V_k$, then $\tilde\Phi_k=(\Phi_k,id)$.
It is clear that $\tilde\Phi_k$ is an embedding around $\tilde\Sigma_k$ and the pull-back of the metric $\bar g_{k+1}$ is
\begin{equation*}
\tilde\Phi_k^*(\bar g_{k+1})=\phi^2\mathrm dt^2+\tilde\Phi_{k,t}^*(\tilde g_{k,t})
=\phi^2\mathrm dt^2+\Phi_{k,t}^*(g_{k,t})+\sum_{i=k+1}^{n+1}(u_i|_{\Sigma_{k,t}})^2\mathrm dt_{i+1}^2,
\end{equation*}
where $\phi>0$ is the lapse function, $\tilde g_{k,t}$ and $g_{k,t}$ are the induced metrics of $\tilde\Sigma_{k,t}$ and $\Sigma_{k,t}$, respectively.
Since $\tilde\Sigma_{k,t}$ is totally geodesic, we have $\partial_t\tilde\Phi^*(\tilde g_{k,t})=2\phi \tilde h_{k,t}=0$, which yields $\tilde\Phi^*(\tilde g_{k,t})\equiv \tilde g_k$ and that $u_i|_{\Sigma_{k,t}}=u_i|_{\Sigma_k}$ are positive constants for $i$ from $k+1$ to $n+1$.
Also, from stability we have $\tilde{\mathcal J}\phi=-\Delta \phi=0$, which gives $\phi(\cdot,t)\equiv const$. Let
$$r(t)=\int_0^t\phi(\cdot,s)\,\mathrm ds,$$
then $\tilde\Phi_k^*(\bar g_{k+1})=\mathrm dr^2+\tilde g_k$. Therefore, $\tilde\Phi_k:\tilde\Sigma_k\times (-\epsilon,\epsilon)\to\bar\Sigma_{k+1}$ is a local isometry. Through a continuous argument as in \cite[Proposition 11]{BBN2010}, we conclude that there exists a local isometry $\tilde\Phi_k:\tilde\Sigma_k\times\mathbf R\to\bar \Sigma_{k+1}$. As a result, we obtain $R_{\bar g_{k+1}}=R_{\tilde g_k}=2$ and that $u_i|_{\Sigma_{k+1}}$ are positive constants for $i$ from $k+1$ to $n+1$. Now, the proposition follows easily from the induction argument.
\end{proof}

From the proof above, we have the following corollary:

\begin{corollary}\label{Cor: rigidity slicing}
If $\mathcal A_{\mathbf S^2}(M,\bar g)=4\pi$, then for any slicing constructed as in Proposition \ref{Prop: slicing}, there exist local isometries $\Phi_k:\Sigma_k\times \mathbf R\to \Sigma_{k+1}$ for $k=2,3,\ldots,n+1$.
\end{corollary}
We now prove the main theorem.
\begin{proof}[Proof of Theorem \ref{Thm: main1}]
From Corollary \ref{Cor: area-bound}, we have $\mathcal A_{\mathbf S^2}(M,\bar g)\leq 4\pi$. If the equality holds,
from Proposition \ref{Prop: slicing}, Lemma \ref{Lem: tilde Sigma 2} and Corollary \ref{Cor: rigidity slicing}, we can find a slicing
\begin{equation*}
(\Sigma_2,g_2)\looparrowright(\Sigma_3,g_3)\looparrowright\cdots\looparrowright (\Sigma_{n+1},g_{n+1}) \looparrowright (\Sigma_{n+2},g_{n+2})=(M,\bar g)
\end{equation*}
such that $(\Sigma_2,g_2)$ is isometric to the standard sphere and that there exist local isometries $\Phi_k:\Sigma_k\to \Sigma_{k+1}$ for $k=2,3,\ldots,n+1$.
Denote $\Phi$ to be the composition of the following maps
\begin{equation}
\mathbf S^2\times \mathbf R^n=\Sigma_2\times\mathbf R^n\stackrel{(\Phi_2,id)}\longrightarrow\Sigma_3\times\mathbf R^{n-1}\stackrel{(\Phi_3,id)}\longrightarrow\cdots\stackrel{\Phi_{n+1}}\longrightarrow \Sigma_{n+2}=M,
\end{equation}
then $\Phi$ is a local isometry from $\mathbf S^2\times \mathbf R^n$ to $M$. Since $\mathbf S^2\times \mathbf R^n$ is complete, $\Phi$ is a covering map. It is clear that $\mathbf S^2\times \mathbf R^n$ is the universal covering of $(M,\bar g)$.
\end{proof}

\bibliography{RigiditySphere}
\bibliographystyle{amsplain}

\end{document}